\documentclass{amsart}
\usepackage{amssymb,latexsym}
\usepackage{graphicx}
\newtheorem{theorem}{Theorem}
\newtheorem{proposition}{Proposition}
\newtheorem{lemma}{Lemma}
\newtheorem{conjecture}{Conjecture}
\newtheorem{definition}{Definition}

\newcommand{\join}{\ensuremath{\vee}} 
\newcommand{\meet}{\ensuremath{\wedge}} 

\newcommand{\mb}[1]{\ensuremath{\mathbf{#1}}} 
\newcommand{\rgf}[1]{\ensuremath{\mathbf{rgf}({#1})}} 

\newcommand{\h}[1]{h(#1)}    

\newcommand{\bin}[3]{\ensuremath{\binom{#1}{#2}}_{#3}}
\newcommand{\faq}[1] {\ensuremath {(#1)!_q}}

\newcommand{\hide}[1]{}
\newcommand{\hidecompletely}[1]{}
\newcommand{\bo}{\mb b}
\newcommand{\tp}{\mb t}

\begin{document}
\title{The Tutte $q$-Polynomial}
\author{Guus Bollen, Henry Crapo, Relinde Jurrius}
\date{\today}

\maketitle

\section{Introduction}

The concept of a {\it $q$-matroid} has recently been developed in the context of coding theory by Relinde Jurrius and her colleague Ruud Pellikaan [JP]. As is the case for matroids, defined by Hassler Whitney, there are a fair number of equivalent axiom systems for $q$-matroids. Gian Carlo Rota introduced the adjective  {\it cryptomorphic} to describe such equivalence of axiom systems, since the axioms are phrased in terms of {\it different} fundamental objects and measures:  rank, independent sets, bases, circuits, copoints, and other related structures, each being a sort of ``code'' for the others. Hassler Whitney's first paper [W1] on this subject was devoted to proving the equivalence of these cryptomorphic axiom systems for matroids.

The only essential difference between matroids and $q$-matroids is in the choice of a different type of {\it support lattice} to replace the Boolean algebra of subsets of the set of matroid elements (including loops and points, both simple and multiple). For $q$-matroids, the support lattice is a (more general) complemented modular lattice. In this paper we will restrict our attention to support lattices $L_{q,k}$  of height $k$ of subspaces of a vector space over the finite field $GF(q)$, where $q$ is a prime power $q=p^n$, these being the lattices of flats of projective geometries of dimension $k-1$. We reserve the term ``rank'' for a measurement made {\it within} a $q$-matroid, so we will call this value $k$ the {\it height} of the support lattice $L$, written $\h{L} = k$, the length of any maximal chain $x_0<x_1<\dots<x_k$ of flats in $L$. The flats of these lattices can be coordinatized by Grassmann coordinates calculated over $GF(q)$. 
\hide{({\it Question: are support lattices $L_{q,k}$ well defined for $q$ not a prime power? If so, is there a sensible "coordinatisation" of the flats?}.)}

The lattice elements of a support lattice $L$, these being subspaces of a vector space $V=L_{q,k}$ we will call {\it support flats}. (The adjective will be essential, because we wish also to refer to the ``flats'' of the $q$-matroid constructed upon the support lattice $L$, these being the {\it closed} support flats (see below). As for classical matroids, support flats of height $1$ are {\it elements} of the associated $q$-matroid, and can be loops, points, or even elements of multiple points in the $q$-matroid.

So the support lattices $L=L_{q,k}$ of $q$-matroids will here be not only complemented modular, but also irreducible: they will have no decomposition as a direct sum of smaller lattices\footnote{Common properties of matroids and $q$-matroids are  easily proven more generally for support lattices that are complemented modular, or even merely modular. See the last section of this article.}.

In their paper [JP],  Jurrius and Pellikaan chose first to define $q$-matroids in terms of the concept of a {\it rank function}. 
\begin{definition}
A {\it$q$-matroid} $M$ is a pair $(L,\rho)$, in which $L=L_{q,k}$ is a finite-dimensional vector space over a field $\mathbb{F}$, and $\rho$ an integer-valued function defined on the set of subspaces of $L$, called the {\it rank}, such that for all subspaces $A,B \; \in L$,
\begin{enumerate}
\item[$R_1$\quad\quad] $0\;\leq\;\rho(A)\;\leq\;\h{A}$. ({\it bounded by height})
\item[$R_2$\quad\quad] If $A\subseteq B$, then $\rho(A)\leq \rho(B)$. ({\it increasing})
\item[$R_3$\quad\quad] $\rho(A + B) + \rho(A\cap B) \;\leq\;\rho(A) + \rho(B)$. ({\it submodularity})
\end{enumerate}
\end{definition}
\noindent this being a straight-forward analogue of the definition of a classical matroid (on a set) in terms of its rank function. 

Perhaps the most intuitive definition of $q$-matroids is that they are {\it strong map images} of support lattices $L_{q,k}$, bearing in mind that their support lattices are themselves geometric, and that the {\it image} of any strong map between geometric lattices is itself geometric. Recall that a lattice is {\it geometric} if and only if it is semimodular, atomistic, and relatively complemented, and that a {\it strong map} $\sigma: G\to H$ of a geometric lattice $G$ (i.e.: of the lattice of flats of a matroid) to a geometric lattice $H$ is any map that preserves joins, and preserves the relation $\downarrow$, where $y\downarrow x$ if and only if $y$ covers or is equal to $x$.  

Since the rank function of a $q$-matroid $M$ is semimodular, with increase by either $0$ or $1$ on every covering in $L$, the lattice $L_M$ of closed flats in $L$ is geometric, and is the image of the strong map $\sigma: L \twoheadrightarrow L_M$.

\begin{theorem} 
A lattice $L_M$ is the lattice of closed flats of a $q$-matroid defined on a support lattice $L_{q,k}$ if and only if $L_M$ is the image of the projective geometry $L_{q,k}$ under a strong map $L_{q,k} \twoheadrightarrow L_M$ onto $L_M$.
\end{theorem}

Given a strong map $\sigma: L_{q,k} \twoheadrightarrow L_M$ for a $q$-matroid $M$, we may precede $\sigma$ by the strong map $\pi:B(P) \rightarrow L_{q,k}$, where $P$ is the set of points of the projective geometry $L_{q,k}$, and $B(P)$ the Boolean algebra of subsets of $P$. We obtain a composite strong map, also onto $L_M$,  having the same geometric lattice as image, and thus a classical matroid $C(L_M)$ with the same geometric lattice of closed sets. We will refer to this matroid and the {\it associated classical matroid} of $L_M$. (See [JP], bottom of page 5.)

As a simple example, consider the uniform $q$-matroid $U_{2,1}$ (rank 2, nullity 1) on the Fano geometry $L_{2,3}$. There are seven closed flats of rank 1 in $U_{2,1}$. The image is a seven-point line, which is not binary, that is, it is not representable over the prime field $GF(2)$.

So the question ``Which matroids $M$ are $q$-matroids?'', in the sense that $M$ is the associated classical matroid $C(M_q)$ of some $q$-matroid $M_q$, is in some sense complementary to the question ``Which matroids $M$ are $q$-representable?''. $q$-Matroids only become a minor-closed class if we use a definition of ``minor'' that is natural for  their support lattices $L=L_{q,k}$, that is, restriction  to an interval of $L$, such an interval also being  the lattice of flats of a projective geometry. But this set of minors is a {\it proper} subset of the set of minors of the associated classical matroid $C(M_q)$. 

\section{Axioms in terms of bi-colorings of support lattices}

In this paper we find it convenient to use yet one more cryptomorphic system of axioms for matroids, one which extends with slight alteration to $q$-matroids. Let $L$ be a support lattice of finite height $h$. The elements of $L$ we call {\it flats}. Support flats of height 1 are its {\it atoms}. The ``zero'', or bottom flat of the support lattice we denote $\bo$; the top flat of height $\h{L}$ we denote $\tp$.

The intervals of length 1 we call {\it covers}. Intervals of length 2 we call {\it diamonds}. A convenient axiomatization of $q$-matroids on a support lattice $L$ is expressed in terms of a bi-coloring (red, green) of the set of coverings in the lattice $L$. A bi-coloring is {\it matroidal} if and only if each diamond is of one of four types:
\vbox{
\begin{enumerate}
\item[One --] All coverings are red
\item[Mixed --] Exactly one lower covering is green, and the covering above it is the only upper red covering 
\item[Prime --] All lower coverings are red, all upper coverings are green
\item[Zero --] All coverings are green
\end{enumerate}
}

\begin{figure}[h] 
\centering\includegraphics[scale=.64]{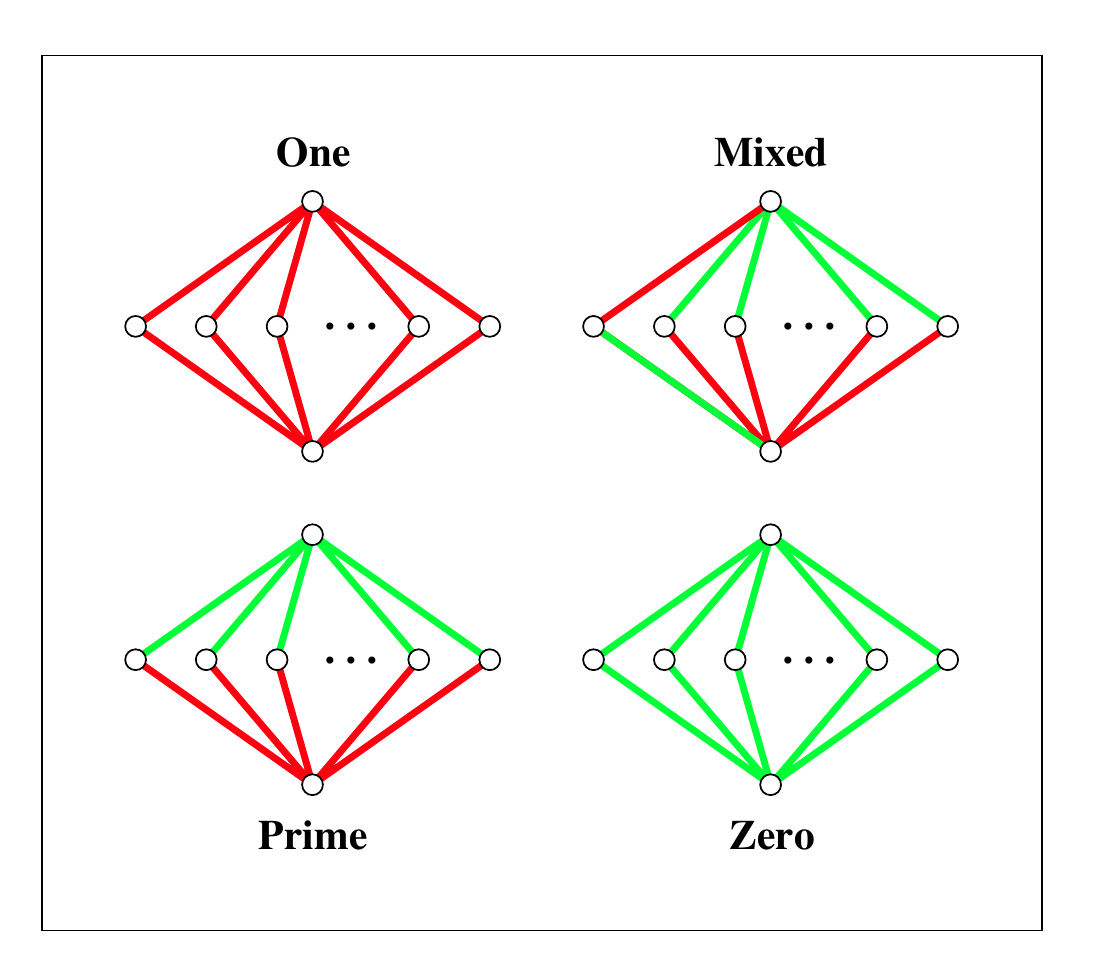} 
\caption{The four diamonds in a $q$-matroid.}
\end{figure}

If $x$ and $y$ are two support flats, let $z = x\meet y$ and $w = x\join y$. Since the support lattice is modular, the following are equivalent:
$$ y \text{\;covers\;}z = x\meet y\quad\iff\quad w = x\join y\text{\;covers\;} x.$$
In this case we say that the cover
$[z,y]$ {\it projects upward} to $[x,w]$, and that the cover
$[x,w]$ {\it projects downward} to $[z,y]$.
In such a case, let 
$$z = x\meet y = x_0\;<\;x_1\;<\;\dots \;<x_{n-1}\;<\;x_n = x$$
 be a maximal chain of flats from $z$ to $x$. Then setting $y_i = y\join x_i$ for $i = 0,\dots,n$,
$$y = y_0\;<\;y_1\;<\;\dots \;<y_{n-1}\;<\;y_n = x\join y = w$$
is a maximal chain from $y$ to $w$.  Dually, if
$$y =  y_0\;<\;y_1\;<\;\dots \;<y_{n-1}\;<\;y_n = x\join y$$
 is a maximal chain of flats from $y$ to $x\join y$. Then setting $x_i = x\meet y_i$ for $i = 0,\dots,n$,
$$x\meet y = x_0\;<\;x_1\;<\;\dots \;<x_{n-1}\;<\;x_n = x$$
is a maximal chain from $x\meet y$ to $x$. In both cases, $y_i$ covers $x_i$ for all indices $i$ in the range $0\leq i \leq n$. In this case we say these $2(n+1)$ flats $x_i,y_i$ form a {\it ladder} from covering $[z,y]$ to  covering $[x,w]$,  with $[z,y]$ projecting upward to $[x,w]$, and 
$[x,w]$ projecting downward to $[z,y]$.

The axioms for the matroidal bi-coloring of the support lattice imply the following two logically equivalent statements, but they do not suffice to define matroids: 
\vbox{
\begin{enumerate}
\item[(1.)\quad] If a covering $[z,y]$ projects upward to a covering $[x,w]$, and $[z,y]$ is green, so is $[x,w]$;
\item[(2.)\quad] If a covering $[x,w]$ projects downward to a covering $[z,y]$, and $[x,w]$ is red, so is $[z,y]$.
\end{enumerate}
}

To prove that the rank axioms and bi-coloring axioms for $q$-matroids are cryptomorphic, we will need to know that for any $q$-matroid bicoloring and for every interval $[x,y]$ in $L$, the number of red coverings in a maximal chain from $x$ to $y$ is invariant, independent of the choice of path.

\begin{proposition} 
For a $q$-matroid defined in terms of a bicoloring of a support lattice $L$, the value $\rho(x,y)$ is well-defined, independent of choice of maximal chain, for any ordered pair of flats $x\leq y$ in $L$
\end{proposition}

\begin{proof}
It is clear that the number of red coverings in a maximal chain of flats from $x$ to $y$ is invariant when $\lambda(y)-\lambda(x)$ is equal to 0 or 1, since there is then only one maximal chain from $x$ to $y$. Under the induction assumption that this is true for all intervals $[c,d]$ of length less than $k$, let $[x,y]$ be an interval of length $k$ in $L$, and let
$$
\text{chain}\; P:\;\;c = x_0,x_1, \dots,x_k=d\text{\quad and\quad}\text{chain}\; S:\;\;c = y_0,y_1, \dots,y_k=d
$$
be any two maximal chains from $c$ to $d$. Then let $z_2$ be the support flat $x_1\join y_1$, and let $z_2,z_3,\dots,z_k=d$
be a maximal chain from $z_2$ to $d$, so that 
$$\text{chain}\; Q:\;\;c,x_1,z_2,z_3,\dots, z_k \text{\quad and \quad}\text{chain}\; R:\;\; c,y_1,z_2,z_3,\dots, z_k$$
are also maximal chains from $c$ to $d$. By the induction hypothesis, chains $P$ and $Q$ both have 
$$\rho(c,x_1) + \rho(x_1,d)\;\;=\;\;\rho(c,x_1) + \rho(x_1,z_2) + \rho(z_2,d)$$
red coverings, and chains $R$ and $S$ both have 
$$\rho(c,y_1) + \rho(y_1,d)\;\;=\;\;\rho(c,y_1) + \rho(y_1,z_2) + \rho(z_2,d)$$
red coverings. By the axiomatic restriction on bi-coloring of diamonds,  
$$\rho(c,x_1) + \rho(x_1,z_2) \;\;=\;\;\rho(c,y_1) + \rho(y_1,z_2),$$
so chains $Q$ and $R$ have the same number of red coverings. Therefore chains $P$ and $S$ have the same number of red coverings.
\end{proof}

So the concept of {\it rank} $\rho(x)$ of a support flat $x$ is well-defined as the number of red coverings in any  maximal chain from $\bo$ to $x$.

\begin{figure}[h] 
\centering\includegraphics[scale=.64]{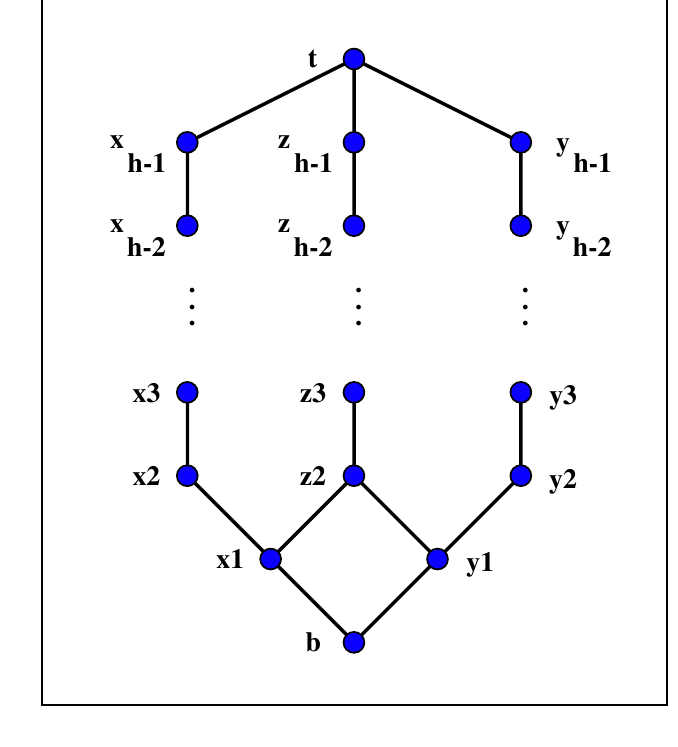} 
\caption{To show by recursion that the rank function is well-defined.}
\end{figure}

\begin{theorem} 
The axiom systems for $q$-matroids in terms of rank and in terms of matroidal bi-colorings are cryptomorphic. 
\end{theorem}
\begin{proof}
Assume that we have a $q$-matroidal bi-coloring of a support lattice $L$. For any support flat $x$, let the value $\rho(x)$ be the number of red coverings in any maximal chain of support flats from $\bo$ to $x$. This value $\rho(x)$ is clearly bounded by height, and increasing, so axioms $R_1$ and $R_2$ hold. Now let $x$ and $y$ be arbitrary support flats in $L$. Let $z=x\meet y$ and $w=x\join y$, and let $z= x_0,x_1,\dots x_m=x$ be a maximal chain from $x\meet y$ to $x$, so 
setting $y_i = y\join x_i$ for $0\leq i\leq m$ we find that the covering $[x_i,x_{i+1}]$ projects upward to the covering $[y_i,y_{i+1}]$
for all $i$ in the interval $0\leq i < m$, so the number of red coverings in the interval $[y,x\join y]$ is bounded above by the number of red coverings in the interval $[x\meet y,x]$. Thus 
$$\rho(x\join y)\;-\;\rho(y)\;\;\leq\;\;\rho(x)\;-\;\rho(x\meet y),$$
and the proposed rank function $\rho$, defined in terms of the bi-coloring of the support lattice, is semi-modular.

For the converse, assume that $M$ is a $q$-matroid defined via a rank function $\rho$ obeying axioms $R_1,R_2,R_3$. Bi-color the support lattice $L$ of $M$, making red all coverings along which there is an increase in rank. Let $[x,z]$ be an interval of height 2 in $L$ (a ``diamond''). Let $d=\rho(z)-\rho(x)$. Since $\rho$ is bounded by height and increases by at most 1 on each covering, $0\leq d\leq 2$. If $d=0$ or $d=2$, the diamond is of type ``zero'' or ``one'', respectively. For $d=1$, since $\rho$ is semimodular, the number of green coverings $[x,y]$ with $x\leq y\leq z$ is either $1$ or $q+1$, so the diamond in question is of type ``mixed'' or ``prime'', respectively.
\end{proof} 


 Jurrius and Pellikaan  proposed a cryptomorphic axiom system for $q$-matroids in terms of  {\it independent} flats, that is, in terms of support flats $A$ for which $\rho(A) = h(A)$. 
Their axioms:
\begin{enumerate}
\item[$I_1$\quad\quad] The zero flat is independent.
\item[$I_2$\quad\quad] Any subflat of an independent flat is independent.
\item[$I_3$\quad\quad] If $\h{A}>\h{B}$ for independent flats $A,B$, then there is an atom $a$ in $A$ but not in $B$ such that $B\join a$ is independent.
\item[$I_4$\quad\quad] For support flats $A,B$ with, respectively, maximal independent subspaces $I,J$, then there is, within $I\join J$, a maximal independent subspace of $A\join B$.
\end{enumerate}
Below, we will comment on the axiom $I_4$, which is not required in the corresponding definition of classical matroids.

Some familiarity with these structures can be gained by studying the following drawing (Figure 3) of the eight possible $q$-matroids definable on the support lattice of  height  3 over the two-element field $GF(2)$.
The underlying geometric figure is the Fano matroid, with 7 points and 7 lines, one line drawn as that inner circle. The outer circle stands for the entire projective plane, the top flat in the lattice $L_{2,3}$. All independent flats are drawn in blue, dependent flats in yellow.


\begin{figure}[h] 
\centering\includegraphics[scale=.64]{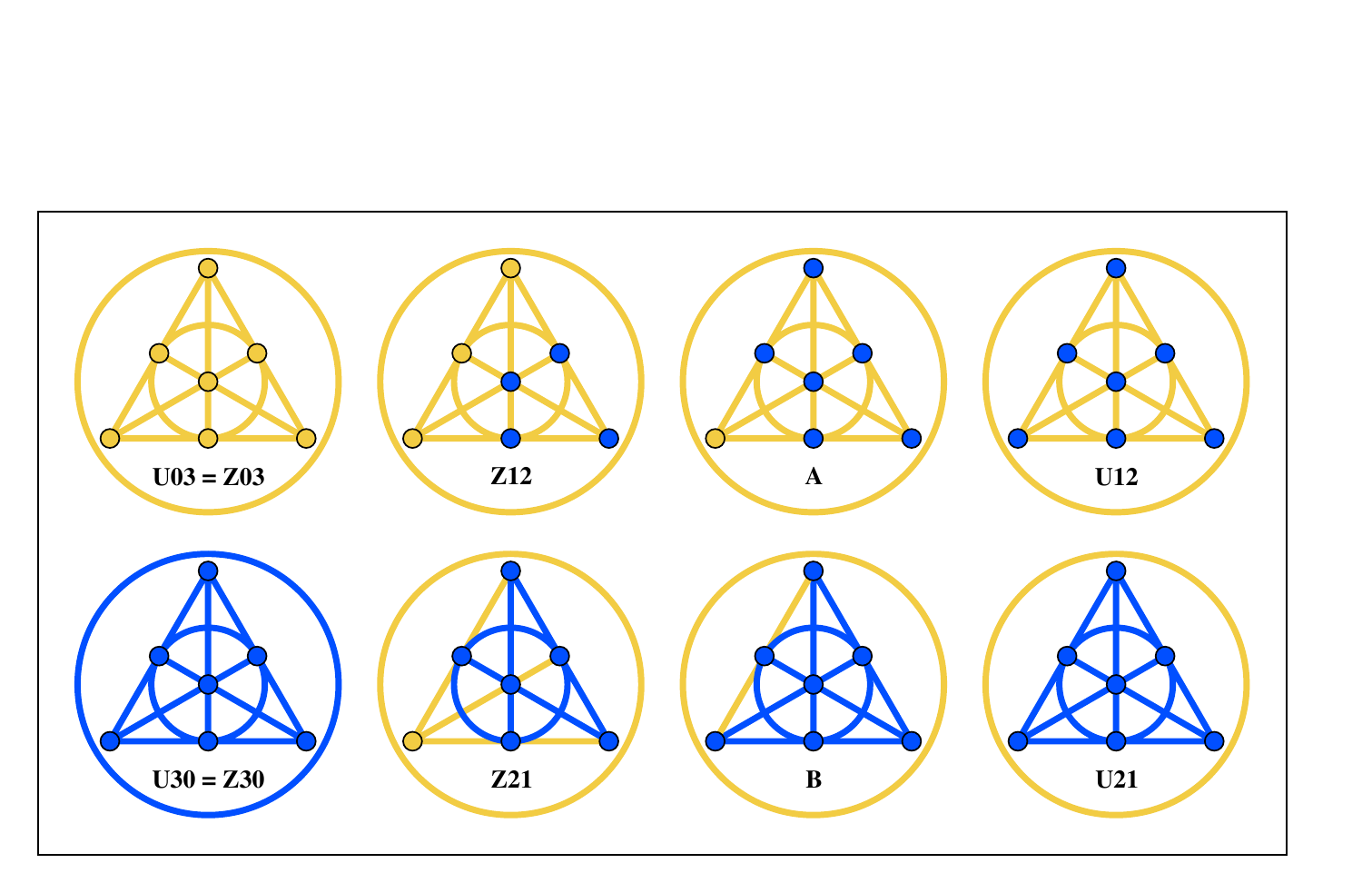} 
\caption{The eight $q$-matroids definable on support lattice $L_{2,3}$, up to isomorphism. Independent = blue, dependent = yellow.}
\end{figure}

 Axiom $I_4$ in the list of  axioms for independent support flats in a $q$-matroid is not
needed in the corresponding axiomatization of classical matroids. Why is it necessary for $q$-matroids, as determined by a bi-coloring? The diagram in Figure 4 shows a support lattice interval of height 2, for $q\geq 2$, which obeys axioms $I_1,I_2,I_3$, but  which is not allowed in a $q$-matroid. It is excluded by axiom $I_4$, since the two atoms on the left have the zero flat as maximal independent subspace, but their join, the top flat, has no maximal independent subspace within the join of the zero flat with itself. This problem does not occur in classical matroids.

\begin{figure}[h] 
\centering\includegraphics[scale=.64]{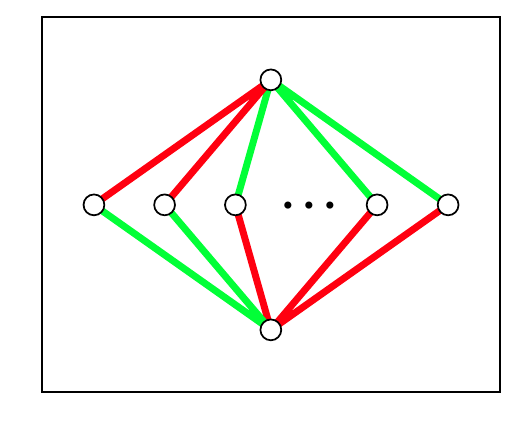} 
\caption{A non-matroidal diamond.}
\end{figure}


\section{Minors of $q$-matroids}


For any two support flats $z,w$ such that $z\leq w$ in $L$, the restriction of the matroidal bi-coloring
 $M$ to the lattice interval $[z,w]$ satisfies the axioms for a matroidal bi-coloring. The resulting $q$-matroid is called a {\it minor} of $M$, denoted $M|_{[z,w]}$. Two special cases of  particular interest:
a minor on a lower interval $[\bo,x]$, $\bo$ being the bottom element of the support lattice, is called the {\it restriction to}  flat $x$. A minor on an upper interval $[x,\tp]$, $\tp$ being the top element of the support lattice, is called the {\it contraction by}  flat $x$.

The {\it uniform} $q$ matroid of rank $\rho$, nullity $\nu$ on a support lattice $L$ of height $h=\rho+\nu$ is given by the bi-coloring that is red on all coverings $[x,y]$ for flats $y$ of height $\h{y}\leq r$, and green on all higher coverings. We take the liberty of using an unconventional notation $U_{\rho,\nu}$ for this uniform matroid \footnote{The conventional notation $U_{\rho,\rho + \nu}$ would make many of our subsequent formulations be more difficult to digest, particularly with respect to our notation for prime-free $q$-matroids.}. Note that in a uniform $q$-matroid, no  diamonds are ``mixed''. 

At the other extreme from uniform matroids, we find those $q$-matroids that have no ``prime'' diamonds. On a given support lattice $L$ there is, up to isomorphism, a doubly-indexed family of {\it prime-free} $q$-matroids that we denote $Z_{i,j}$, of rank $i$ and nullity $j$. 

\begin{figure}[h] 
\centering\includegraphics[scale=.64]{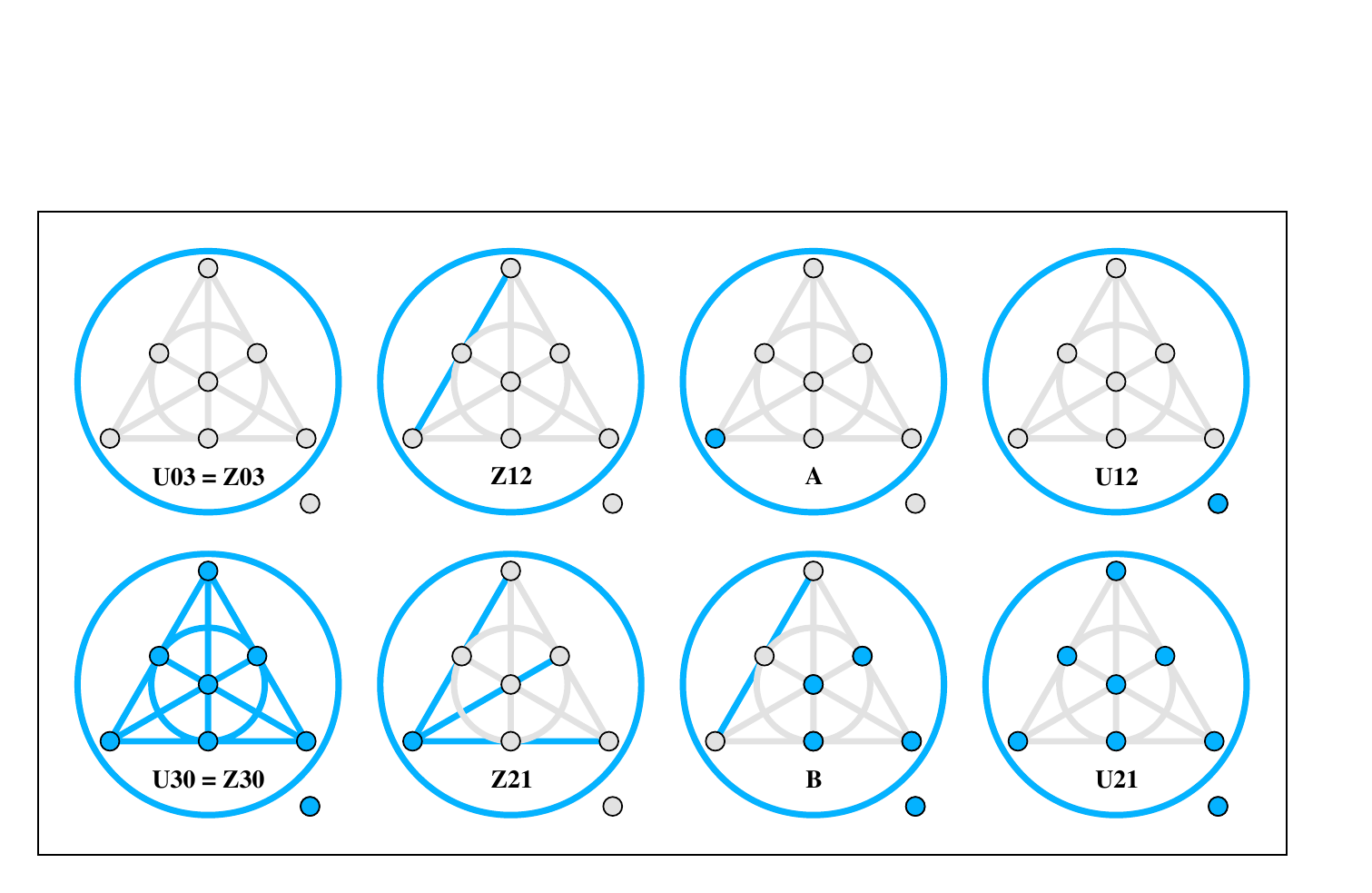} 
\caption{The closed flats, marked in  blue, for eight $q$-matroids defineable on support lattice $L_{2,3}$, up to isomorphism. The small circle, offset lower right in each case, marks the zero subspace of $L$.}
\end{figure}

A support flat $x$ in a $q$-matroid $M$ is {\it closed}  if and only if for all support flats $y$ covering $x$, the covering $[x,y]$ is red. A closed flat $x$ is also called a {\it flat of $M$}, despite the likely confusion in terminology.

A support flat $x$ in a $q$-matroid $M$ is {\it open}  if and only if for all support flats $z$ covered by  $x$, the covering $[z,x]$ is green. 

A flat $x$ that is both closed and open is called {\it clopen}. Each  prime-free matroid of nullity $j$ has a {\it unique} clopen flat of height $j$, nullity $j$, and rank $0$. Its bi-coloring is completely determined by the choice of that clopen flat. (See Proposition 3, below.)
                                                                                                                                                                                                                                                                
A {\it circuit} in a $q$-matroid $M$ is defined to be a support flat $c$ such that the corresponding restriction $M|_{[\bo,c]}$ is uniform of rank $\h{c}-1$. The circuits of a $q$-matroid are its minimal proper (not equal to the bottom flat $\bo$) open flats. They have nullity 1.

A {\it copoint} in a  $q$-matroid $M$ is defined to be a support flat $d$ such that the corresponding contraction $M|_{[d,\tp]}$ is uniform of rank $1$. The copoints of a $q$-matroid are its maximal proper (not equal to the top flat $\tp$) closed flats.

A support flat $x$ in a $q$-matroid $M$ is a {\it spanning flat}  if and only if for all support flats $y$ covering $x$, the covering $[x,y]$ is green.

A support flat $x$ in a $q$-matroid $M$ is {\it independent}  if and only if for all support flats $z$ covered by  $x$, the covering $[z,x]$ is red. 

A flat $x$ that is both independent and spanning is called a {\it basis}.

The {\it dual} $M^*$ of a $q$-matroid $M$ on a support lattice $L$ is obtained inverting the bi-colored support lattice for $M$, then interchanging the colors red and green. We say simply that the support lattice for the dual $q$-matroid is the dual lattice $L^{opp}$, its points being the copoints of $L$, that is, the hyperplanes of the associated projective geometry.

\begin{proposition} 
In a modular lattice $L$, for any lattice element $z$ and any covering pair $[x,y]$, exactly one of the following statements is true:
$$\begin{array}{c}
y\join z \text{ covers } x\join z \text{ and } y\meet z\;=\;x\meet z\\
y\join z\;=\;x\join z \text{ and } y\meet z \text{ covers } x\meet z\\
\end{array}$$
\end{proposition}
\begin{proof}
In a modular lattice $L$, for any element $z\in L$ and any covering $[x,y]$, it is not the case that both $y\leq x\join z$ and $x\geq y\meet z$, so either $y\join z$ covers $x\join z$ or $y\meet z$ covers $x\meet z$. But only one of these two conclusions can hold, since if $y\join z$ covers $x\join z$ then $y\not\leq x\join z$, so $x\leq y\meet (x\join z)< y$ and $x\;=\;y\meet(x\join z)$ so $x\meet z\;=\;y\meet (x\join z)\meet z\;=\;y\meet z$. Dually, if $y\meet z$ covers $x\meet z$, then $x\join z\;=\;y\join z$.
\end{proof}

\begin{proposition} 
If a $q$-matroid $M$ on a support lattice $L$ has a clopen support flat $z$, then the bi-coloring
is determined as follows: any covering $[x,y]$ in $L$ is green if and only if $y\meet z$ covers $x\meet z$,
and is red  if and only if $y\join z$ covers $x\join z$. In such a $q$-matroid $M$, the flat $z$ is the only clopen flat,  and the complements of $z$ in $L$ are exactly the bases of $M$. 
\end{proposition}
\begin{proof}
Let $M$ be a $q$-matroid with clopen flat $z$.
Since all coverings $[s,t]$ with $t\leq z$ are green, so are all coverings $[x,y]$ up to which $[s,t]$ projects.
Since all coverings $[s,t]$ with $z\leq s$ are red, so are all coverings $[x,y]$ down to which $[s,t]$ projects.
If a flat $x$ in $M$ is clopen, let $y$ be a flat covering $x$, and $c$ a copoint of $L$ such that $c\meet y\;=\;x$. If $c\not\ge z$, then $c\meet z$ is covered by $z$, the covering $[c\meet z, z]$ is green,
as is the covering $[c,\tp]$, up to which it projects. This is not the case, since $c\ge x$ and $x$ is clopen.
So $c\ge z$, and every copoint of $L$ above $x$ is also above $z$. This means that $z\le x$. 
Now let $w$ be any flat covered by $x$, and  $p$  any point of $L$ such that $w\join w\;=\;x$.
By an analogous reasoning, if $p\not\le z$, then $p\join z$ covers $z$, the covering $[z,p\join z]$ is red,
as is the covering $[\bo,p]$ down to which it projects. This is not the case, because $p\le x$ and $x$ is open. So $p\le z$, and every point of $L$ beneath $x$ is also beneath $z$, so $x\le z$, and $x=z$.
\end{proof}

We shall call such a $q$-matroid {\it prime-free}, because it has only three types of diamonds: zero, one, and mixed.

\section{Some $q$-notation}

In Henry Cohn's paper [Co], but without all the square brackets, we have

The $q$-integers:
$$n_q = 1 + q + q^2 + ... + q^{n-1}$$

The $q$-factorials:
$${n!}_q = 1_q  2_q   \dots  n_q$$

The $q$ binomial coefficients:
$${\binom {n} {k} }_q \;=\; \frac   {{n!}_q}   {  {(n-k)!}_q  \;  {{k!}_q}  }$$

\section {The rank generating function}

The {\it rank generating function} $\rgf{M}$ of a $q$-matroid $M$ is defined as it is for matroids. It is the polynomial sum, over all  flats $z$ in the support  lattice $L$, of the monomial $x^{\rho(\tp)-\rho(z)}\;y^{\nu(z)}$. (Recall that $\tp$ is the top flat in $L$.) 

For example, the rank generating function $\rgf{B}$ of the $q$-matroid B in Figure 6 is
$$x^2 + xy + 7x + y + 6, \text{\quad with coefficient array\quad}
\begin{matrix}
1&1\\6&7&1\\
\end{matrix}
$$

\begin{figure}[h] 
\centering\includegraphics[scale=.64]{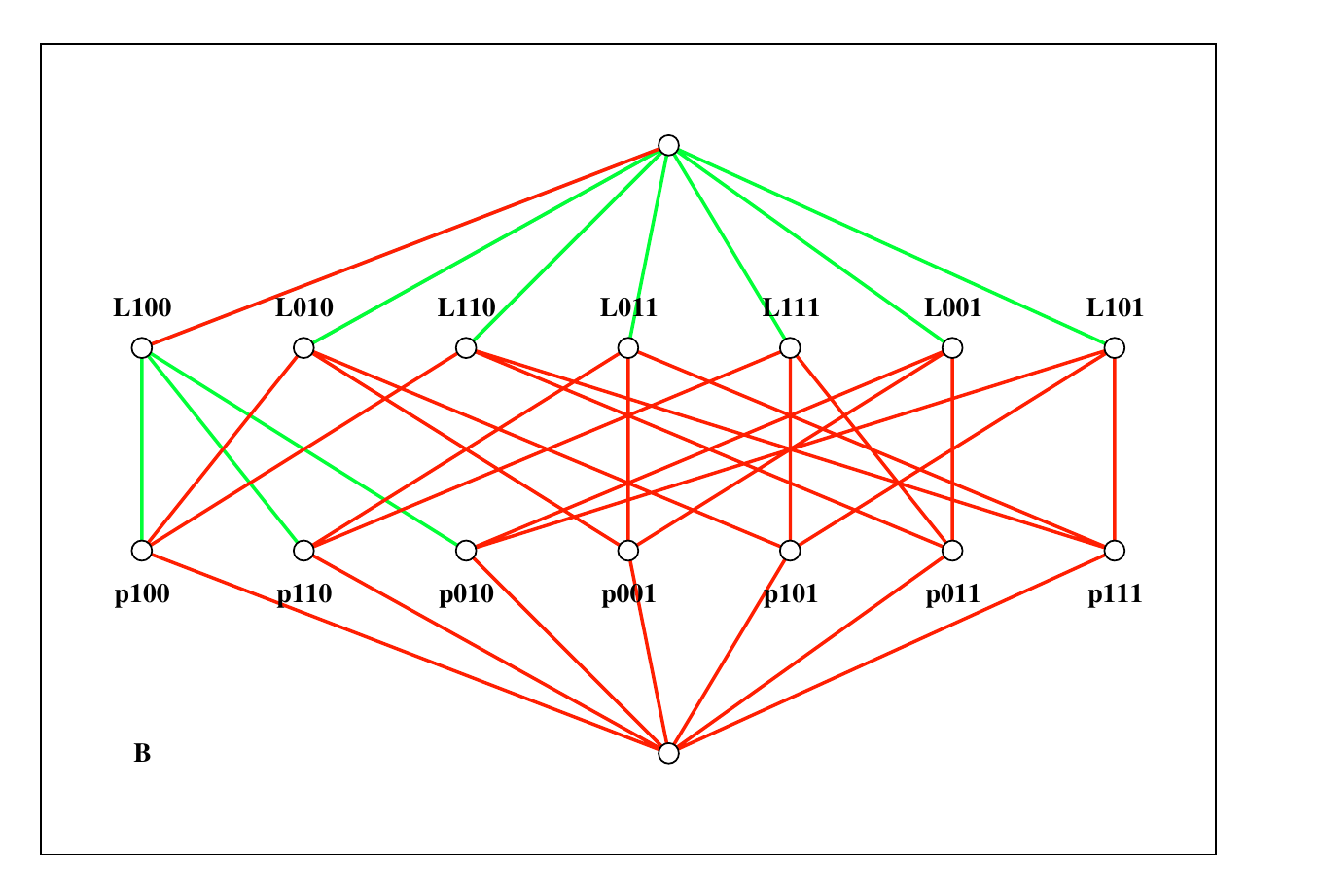} 
\caption{The $q$-matroid B on $L_{2,3}$.}
\end{figure}

In Figure 6 we have labelled each point and line with their Grassmann coordinates. Because $h(L)=3$ and $q=2$, each point will be labelled $p_{a,b,c}$ where $a,b,c$ is a sequence of three binary integers, not all zero. The  join of $A= p_{(a,b,c)}$ and $B= p_{(d,e,f)}$ is the line $A\join B = C$ with coordinates given by the  exterior product $C = L_{(ae-bd,\; af-cd,\; bf-ce)}$, these being the "12", "13", and "23" coordinates of the resulting projective line. For example, $p_{(0,1,0)}\join p_{(1,0,1)}\;=\;L_{(1,0,1)}$. The exterior product
of a point $A= p_{(a,b,c)}$ with a line $D_{(d,e,f)}$ has a single "123" coordinate for the projective plane,
here the top flat $\tp$, equal to $p_1 L_{23} - p_2 L_{13}+ p_3 L_{12}$. This coordinate is non-zero if and only if $p\join D = \tp$, that is, here, if and only if $p$ and $D$ are complementary. If $p<D$, the exterior product is zero, or more precisely, is the zero multiple of $\tp$. Note here that $p_{1,0,0}$ is complementary to a line $D_{(d,e,f)}$ if and only if the coordinate $f=D_{23}$ is non-zero (i.e.: equal to 1). There are four such lines.

\section{Complements in support lattices}

Flats $x$ and $y$ in a support lattice $L$ are {\it complementary} if and only if $x\join y = \tp$ and $x \meet y = \bo$, the top and bottom flats of $L$, respectively. We will show that, for any support lattice $L_{q,h}$, each flat of height $r$ in $L$ has exactly $q^{r(h-r)}$ complements in $L$. Under the assumption that $q=p^n$ for some prime $p$ and integer exponent $n$, so that $L$ is the lattice of flats of a projective geometry over a finite field, we may use the following proof using Grassmann coordinates.

\begin{proposition} 
In a support lattice $L_{q,h}$, each flat $x$ of height $r$ in $L$ has exactly $q^{r(h-r)}$ complements in $L$.
\end{proposition}

\begin{proof}
Let $c$ be a flat of height $r$. By choice of basis, we may assume $c$ is the row space of the following matrix:
\[\left(    \begin{array}{cc}      I_{r} & 0_{r, h-r} \\    \end{array}  \right).\]


Now let $x$ be a complementary flat to $c$. It must have height $h-r$. Then certainly the last $h-r$ columns of the matrix of $x$ are independent. After row reduction, we find that $x$ is the row space of the matrix
\[\left(    \begin{array}{cc}      B & I_{h-r} \\    \end{array}  \right),\]
where $B$ is an $(h-r) \times r$ matrix.

There are $q^{r(h-r)}$ such matrices $B$, each of which yields a different flat $x$ complementary to $c$.
\end{proof}

What follows is a lattice theoretic proof valid for any positive integer value of $q$, including the unsurprising case when $q=1$.

\begin{proof}
We proceed by induction on the height $h$ of $L$. If $0\leq h\leq 1$, the conclusion is a triviality. If $h=2$, the only interesting case is when $r=1$, and each of the $q+1$ atoms has exactly $q$ complements, so the conclusion holds.

If $\h{x} $ is equal to $0$ or $h$, the $x$ has a unique complement, so the conclusion holds.
If $x$ is a coatom, $\h{x} = h-1$, then $L$ has $h_q = 1 + q + q^2 + ... + q^{h-1}$ atoms,
$(h-1)_q = 1 + q + q^2 + ... + q^{h-2}$ of which are beneath $x$, so $x$ has $q^{h-1}= q^{1(h-1)}$
complements, as required.

Otherwise, we may select a coatom $w$ of $L$, with $x<w$. We aim to show that the complements $y$ of $x$ in the overall lattice interval $[\bo,\tp]$ are precisely those support flats $y$ such that the meet $z= y\meet w$
is a complement of $x$ in the interval $[\bo,w]$ {\it and} $y$ is a complement of $w$ in the interval $[z,\tp]$.

The latter condition may be rewriten:
$$
y\meet w = z,\quad y\join w = \tp,\quad x\meet z = \bo,\quad x\join z=w.
$$
If this is the case, then 
$$x\meet y = (x\meet w) \meet y = x \meet (y\meet w) = x \meet z = \bo$$
while $y>z$ and $y\not\leq w$, so $x\join y> x\join z = w$, which implies that $x\join y=\tp$.

Conversely, if $y$ is a complement of $x$, and $z=y\meet w$ then $x\meet z\leq x\meet y=\bo$,
while $x\join z\leq w$ and, by modularity, $x\join y = \tp$ covers $x\join z$, so $x\join z = w$, and
$z$ is a complement of $x$ in the interval $[\bo,w]$. Finally, since $[z,y]$ is a covering projecting upwards to the covering $[w,\tp]$, $y$ is a complement of $w$ in the interval $[z,\tp]$.

By the induction hypothesis, there are $q^{r(h-1-r)}$ distinct complements $z$ of $x$ in the interval $\bo,w$, and for each such support flat $z$, $q^{r}$ complements $y$ of $w$ in the interval $[z,\tp]$, so 
$$
q^{r(h-1-r)}\;\;q^{r}\;\;=\;\;q^{r(h-r)}
$$
complements of $x$ in $L$.
\end{proof}

The following relation (cf. [LW]) among $q$-binomial coefficients is an easy consequence of Proposition 4.

\begin{proposition} 
For all integers $r,h$ with $0\leq r\leq h$,

\begin{equation}     \label{E:curly}
\bin{h}{r}{q}\;\;=\;\;\sum_{i=0}^{r}\;\bin{r}{i}{q}\;\;\bin{h-r}{i}{q}\;q^{(r-i)(h-r-i)}        
\end{equation}

\end{proposition}
\begin{proof}
For any support flat $x$ of height $r$, and for all support flats $c$ of complementary height, there is an integer $i$ such that $\h{x\meet c} =i$ and $\h{x\join c} =h-i$. There are $\bin{h-r}{i}{q}$ choices for $x\join c$, 
$\bin{r}{i}{q}$ choices for $x\meet c$, and $q^{(h-r-i)(r-i)}$ complements $c$ of $x$ in the interval $[x\meet c,x\join c]$, so the number of support flats of height $h-r$, equal to the number of flats of height $r$, is given by the above sum.
\end{proof}

\begin{figure}[h] 
\centering\includegraphics[scale=.8]{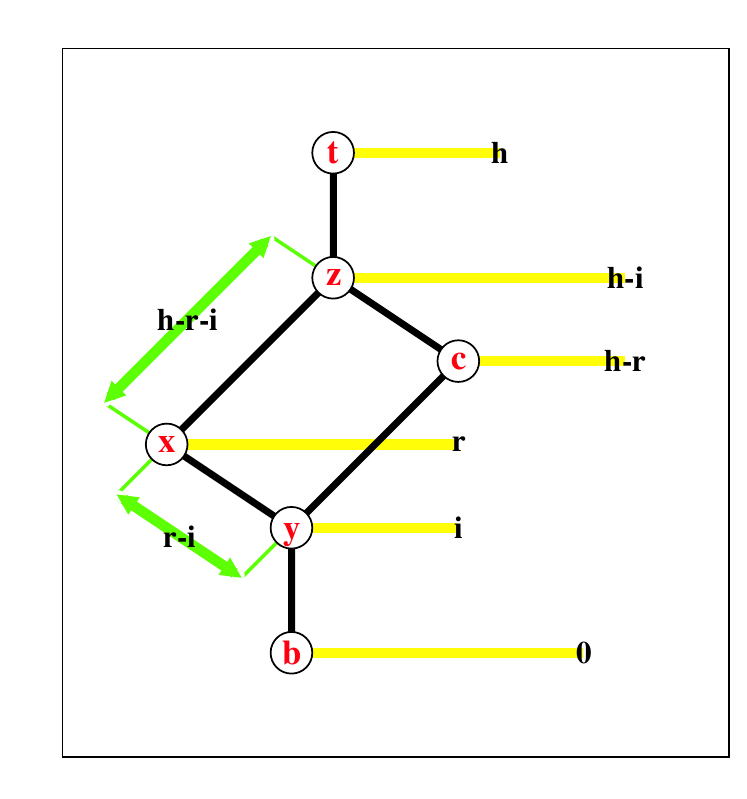} 
\caption{Heights involved in the above proof.}
\end{figure}

\section{Defining the Tutte polynomial}

For a classical matroid $M$ on a set $S$, with the Boolean algebra $2^S$ as support lattice, its Tutte polynomial $\tau_M$, is defined as the sum, over all bases $B$ of $M$, of the monomials $x^iy^j$, where 
$i$ is the {\it internal activity} of the basis $B$, and $j$ is its {\it external activity}. Furthermore, as seen in [C1] and [C3], if one associates with each basis $B$ the minor obtained by restriction to the interval $[B^-,B^+]$, where $B^+\backslash B$ is the set of externally active elements of $B$ and $B\backslash B^-$ is the set of internally active elements of $B$, the resulting minors are prime-free (i.e.: sums of loops and isthmi), each having a single clopen flat. This set of minors $[B^-,B^+]$ {\it partitions} the Boolean support lattice. Let's call this a {\it Tutte partition} of the support lattice.

For $q$-matroids, such partitions of the support lattice into prime-free minors are still present,  each such minor, say of rank $\rho$, nullity $\nu$, having a single clopen flat, and the prime-free minors partition the complemented modular support lattice. But each such minor contains several bases ($q^{\rho\nu}$ bases, to be precise). So the Tutte $q$-polynomial is no longer a two-variable sum over bases, but a two-variable sum over parts of a Tutte partition of the support lattice into prime-free minors.

Each covering within each of these  prime-free minors should be regarded as ``active'',  directed downward if red, upward if green, thus providing directed paths in that minor from all its bases toward the minor's single clopen flat.\footnote{It may seem awkward that internal activity be associated with coverings {\it above} the clopen flat $c$, external activity with coverings {\it below} $c$, in the minor associated with $c$. But this better retains the analogy with the construction of the Tutte polynomial for classical matroids.}.

Let $M_q=Z_{\rho,\nu}$ be a prime-free matroid of rank $\rho$ and nullity $\nu$. Its support lattice $L$ is of height $h=\rho+\nu$.
The matroid $Z_{\rho,\nu}$ will have a unique clopen flat  $c$, which will be of height $\nu$. Any flat $x$ of height 
$\h{x}=k$ in $L$ will have rank lack $h-\h{c\join x}$, nullity $\h{c\meet x}$, and will be a relative complement of $c$ in the interval $[c\meet x,c\join x]$.

As we observed in the previous section, the number of support flats $x$ with these values of rank lack $i$ and nullity $j$ will be a quantity we denote $\alpha_{\rho,\nu;i,j}$, the coefficient of $x^i\;y^j$ in the rank generating function for the prime-free $q$-matroid $Z_{\rho,\nu}$:
\begin{equation}\label{E:zeta}
\alpha_{\rho,\nu;i,j}\;=\;\bin{\rho}{i}{q}\;\bin{\nu}{j}{q}\;q^{(\rho-i)(\nu-j)}
\end{equation}

\begin{theorem} 
The rank generating function of the prime-free $q$-matroid of rank $\rho$, nullity $\nu$ is
\begin{equation}\label{E:TutteToRGF}
\sum_{i=0}^{\rho}\;\sum_{j=0}^{\nu}\;\bin{\rho}{i}{q}\;\bin{\nu}{j}{q}\;q^{(\rho-i)(\nu-j)}\;x^i\;y^j
\end{equation}
So a $q$-matroid $M$ with Tutte $q$-polynomial
$$
\sum_{a=0}^{\rho(M)}\sum_{b=0}^{\nu(M)} \tau_{a,b}\;x^a\; y^b
$$
will have rank generating function 
$$
\sum_{a=0}^{\rho(M)}\sum_{b=0}^{\nu(M)}\;\tau_{a,b}\;\sum_{i=0}^{a}\;\sum_{j=0}^{b}\;\bin{a}{i}{q}\;\bin{b}{j}{q}\;q^{(a-i)(b-j)}\;x^i\;y^j
$$
\end{theorem}

\begin{figure}[h] 
\centering\includegraphics[scale=.75]{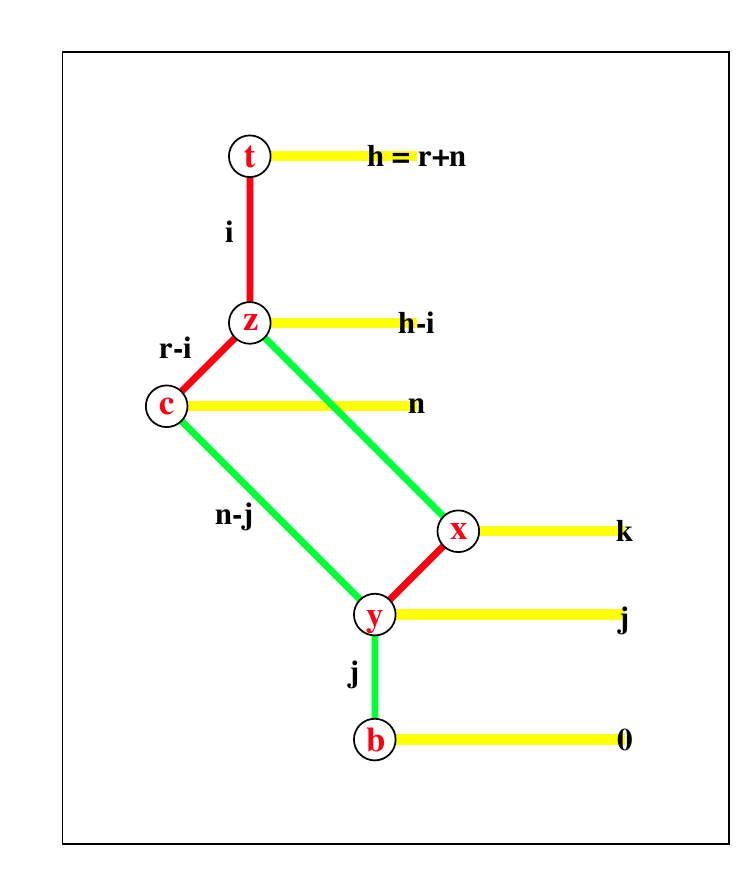} 
\caption{Heights for calculating the {\it rgf}  for the prime-free\quad\quad\quad $q$-matroid $Z_{\rho,\nu}$.}
\end{figure}

This is what replaces the substitution $x \to (x+1),\; \text{and}\; y \to (y+1)$ in the case ``$q=1$'', that is, for ordinary matroids. How are we to transform the rank generating function of a $q$-matroid into its associated Tutte polynomial? What we need is a convolution inverse of the formula $\alpha_{\rho,\nu;i,j}$. This will be a convolution inverse in the context of the incidence algebra of a simple ``rectangular'' poset $R$: the product of two chains $0,\dots,\rho$ and $0,\dots,\nu$. We will call this convolution inverse $\beta_{\rho,\nu;i,j}$. It is an inverse in the sense that, using the symbol $\circ$ to denote convolution,
$$\alpha\circ\beta([a,b],\;[e,f]);=\;\sum_{[c,d]\in R} \alpha(a,b;c,d)\; \beta(c,d;e,f)\;=\;\delta_a^e\;\;\delta_b^f,$$
the term on the right being a product of Kronecker deltas. This is a two-sided inverse, so also 
$$\beta\circ\alpha\; ([a,b],[e,f])\;=\;\delta_a^e\;\delta_b^f.$$

What follows is a conjectured formula for the convolution inverse $\beta$, obtained by studying the results of a Python program that follows the recursive subtraction procedure outlined in Section 8.

\begin{conjecture}
$\beta(a,b;c,d) $ is equal to
$$
 (-1)^{(a-c)+(b-d)}\;\;\bin{a}{c}{q}\;\;\bin{b}{d}{q}\;\;q^{\binom{|(a-c)-(b-d)|}{2}}(1 + q^{|(a-c)-(b-d)|} - q^{max((a-c),(b-d))})
$$
\end{conjecture}

\section{The route back from RGFs to Tutte polynomials}

Say we have the rank generating function of a $q$-matroid $M$ of rank $\rho$, nullity $\nu$. How do we find its Tutte polynomial? 

There is a conceptually simple recursive method to pass from the rank generating function to the Tutte polynomial. If $M$ is prime-free, we know that its Tutte polynomial is the monomial $\tau = x^\rho y^\nu$. Let's start with a simple example, for a $q$-matroid that is not prime-free, the $q$-matroid B seen in Figure 6. The $1$ in position $(1,1)$ in the coefficient array
$$
\begin{matrix}
1&1\\6&7&1\\
\end{matrix}
$$
can only come from a $1$ in that position for the coefficient array of the Tutte polynomial, which would account for values
$$
\begin{matrix}
1&1\\2&1&0\\
\end{matrix}
$$
in the rank generating function. Subtracting off these values, there remains a summand of the rank generating function to explain:
$$\begin{matrix}
1&1\\6&7&1\\
\end{matrix}
\;\;-\;\;
\begin{matrix}
1&1\\2&1&0\\
\end{matrix}
\;\;=\;\;
\begin{matrix}
0&0\\4&6&1\\
\end{matrix}
$$
The $1$ in position $(2,0)$ of the residue can only come from a $1$ in that position in the Tutte polynomial,
which would in turn make a contribution of 
 $$
\begin{matrix}
0&0\\1&3&1\\
\end{matrix}
$$ 
to the rank generating function. Subtracting this contribution, there remains a summand
$$
\begin{matrix}
0&0\\4&6&1\\
\end{matrix}
\;\;-\;\;
\begin{matrix}
0&0\\1&3&1\\
\end{matrix}
\;\;=\;\;
\begin{matrix}
0&0\\3&3&0\\
\end{matrix}
$$
of the coefficient array to explain,
this difference being the coefficient array for {\it 3~times} the rank generating function for $Z_{1,0}$.
We see that the coefficient array for the Tutte polynomial is 
$$
\begin{matrix}
0&1\\0&3&1\\
\end{matrix}
$$

In Figure 9 we see  a Tutte partition for this $q$-matroid $B$, enumerated by exactly the polynomial we have  calculated above.
\begin{figure}[h] 
\centering\includegraphics[scale=.64]{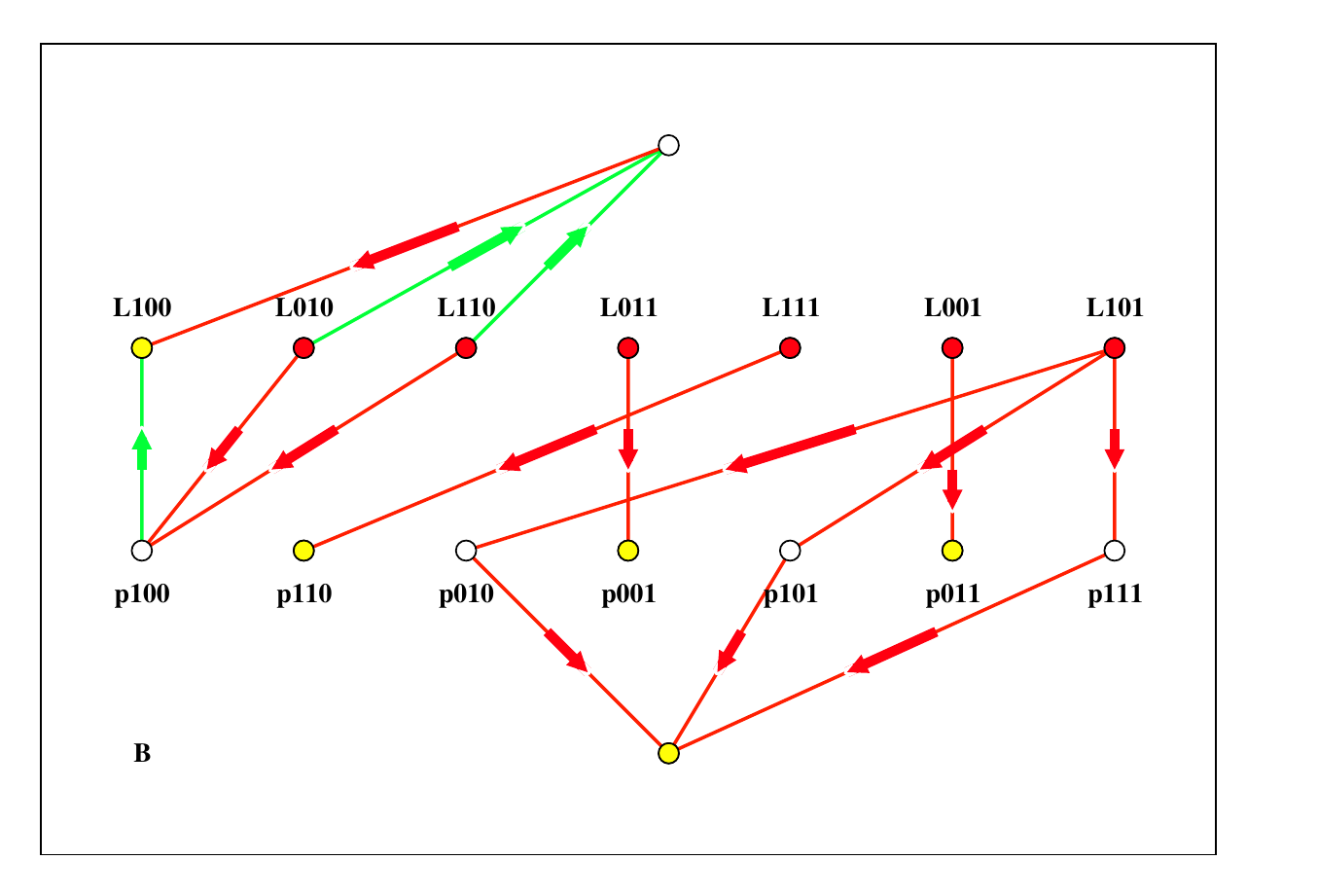} 
\caption{A Tutte partition of $q$-matroid B.}
\end{figure}

In order to see more clearly the role of the $q$-calculus, as in Kac and Cheung [KC], we now consider a well-defined class of $q$-matroids existing {\it for all values of }$q$:  the {\it uniform} $q$-matroids. 
Let's try to find the Tutte polynomial for the uniform $q$-matroid $U_{2,3}$, for all values of $q=p^n$.

We use the notation 
$\tau\;=\;\sum_i\sum_j\;\tau_{i,j}x^iy^j$ and $\rho\;=\;\sum_i\sum_j\;\rho_{i,j}x^iy^j$. 
The rank generating function is
$$\begin{pmatrix}
1\\
1+q+q^2+q^3+q^4\\
1+q+2q^2+2q^3+2q^4+q^5+q^6\\
1+q+2q^2+2q^3+2q^4+q^5+q^6 & 1+q+q^2+q^3+q^4 & 1\\
\end{pmatrix}$$ 

Letting $\tau_{0,3} = 1$, we subtract 
$$\begin{pmatrix}
1\\
1+q+q^2\\
1+q+q^2\\
1&0&0\\
\end{pmatrix}$$ from the rank generating function coefficient array, leaving a remainder of

$$\begin{pmatrix}
0\\
q^3+q^4\\
q^2+2q^3+2q^4+q^5+q^6\\
q+2q^2+2q^3+2q^4+q^5+q^6&1+q+q^2+q^3+q^4&1\\
\end{pmatrix}$$ 

We then set $\tau_{0,2} = q^3+q^4$, and subtract 
$$
\begin{pmatrix}
q^3+q^4\\
(q^3+q^4)(1+q)\\
q^3+q^4&0&0\\
\end{pmatrix} 
\quad=\quad
\begin{pmatrix}
q^3+q^4\\
q^3+2q^4+q^5\\
q^3+q^4&0&0\\
\end{pmatrix}
$$ 

from the previous remainder, leaving
$$\begin{pmatrix}
0\\
0\\
q^2+q^3+q^6\\
q+2q^2+q^3+q^4+q^5+q^6&1+q+q^2+q^3+q^4&1\\
\end{pmatrix}$$ 

We set $\tau_{0,1} =q^2+q^3+q^6$, and subtract  
$$\begin{pmatrix}
0\\
0\\
q^2+q^3+q^6\\
q^2+q^3+q^6&0&0\\
\end{pmatrix}$$
from the previous remainder, leaving
$$\begin{pmatrix}
0\\
0\\
0\\
q+q^2+q^4+q^5&1+q+q^2+q^3+q^4&1\\
\end{pmatrix}$$ 

We set $\tau_{2,0} = 1$, and subtract 
$$\begin{pmatrix}
0\\
0\\
0\\
1&1+q&1\\
\end{pmatrix}$$
from the previous remainder, leaving
$$\begin{pmatrix}
0\\
0\\
0\\
-1+q+q^2+q^4+q^5 &q^2+q^3+q^4&0\\
\end{pmatrix}$$

We set $\tau_{1,0} = q^2+q^3+q^4$, and subtract 
$$\begin{pmatrix}
0\\
0\\
0\\
q^2+q^3+q^4&q^2+q^3+q^4&0\\
\end{pmatrix}$$
from the previous remainder, leaving
$$\begin{pmatrix}
0\\
0\\
0\\
-1+q-q^3+q^5 &0&0\\
\end{pmatrix}$$

Finally, we set $\tau_{0,0} = -1+q-q^3+q^5 $ and collect the coefficients $\tau_{i,j}$ into the coefficient array
$$\begin{pmatrix}
1\\
q^3+q^4\\
q^2+q^3+q^6\\
-1+q-q^3+q^5 &q^2+q^3+q^4&1\\
\end{pmatrix}$$
for the Tutte polynomial.

\section{Converting the RGF to the Tutte polynomial}

To convert a Tutte polynomial to its corresponding RGF, we treat every coefficient $\tau_{\rho,\nu}$ of $x^\rho y^\nu$ in the Tutte polynomial as the number of prime-free minors with rank $\rho$ and nullity $\nu$, so the product 
$ x^\rho y^\nu$ in the Tutte polynomial contributes the terms of Equation (3):
$$
\sum_{i=0}^{\rho}\;\sum_{j=0}^{\nu}\;\bin{\rho}{i}{q}\;\bin{\nu}{j}{q}\;q^{(\rho-i)(\nu-j)}\;x^i\;y^j
$$
So what contribution does each term in the rank generating function make to the Tutte polynomial?

The processes of conversion between rank generating functions and Tutte polynomials are linear over the ring of polynomials in $q$. So it suffices to know the contribution of each monomial $x^a y^b$ in the Tutte polynomial to the corresponding rank generating function, and the contribution of each monomial $x^a y^b$ in the rank generating function to the corresponding Tutte polynomial. The quantity $\alpha(a,b;c,d)$, defined above in formula (3), is the {\it multiplier} of the coefficient of $x^a y^b$ in a Tutte polynomial that produces the coefficient of 
$x^c y^d$ in the associated rank generating function. Inversely, the formula $\beta(a,b;c,d)$, recorded in Conjecture 1, is the {\it multiplier} of the coefficient of $x^a y^b$ in a rank generating function that produces the coefficient of 
$x^c y^d$ in the associated Tutte polynomial. Both $\alpha(a,b;c,d)$ and  $\beta(a,b;c,d)$ are $q$-polynomials.

As we have seen from our study of rank generating functions of the prime-free $q$-matroids $Z_{\rho,\nu}$,
we know that
$$
 \alpha(a,b;c,d)\;=\;\bin{a}{c}{q} \bin{b}{d}{q} q^{(a-c)(b-d)}. 
$$
This formula agrees with the formula for the classical case, where the conversion is accomplished by the substitution $x\to(x+1), y\to(y+1)$, since $\bin{i}{j}{1}= \binom{i}{j}$ and $1^{(a-c)(b-d)}=1$.
\clearpage

Using the algorithm of recursive subtraction, as in Section 8, programmed in Python, we arrive at Conjecture 1, above, that the contribution
$\beta(a,b;c,d)$ of a coefficient 1 of $x^a y^b$, in a rank generating function, to the coefficient of $x^c y^d$ in the corresponding Tutte polynomial
\begin{enumerate}
\item[(1)] has an alternating sign $S= (-1)^{(a-c)+(b-d)}$
\item[(2)] has a factor of $\bin{a}{c}{q}$
\item[(3)] has a factor of $\bin{b}{d}{q}$
\item[(4)] has a final term Q(a,b;c,d), which we now  explain.
\end{enumerate}

This final term $Q$ is particularly interesting. The product of the sign $S$ with $Q$ is none other than the "scalar" or "corner" coefficient for the rank generating function of the "complementary" prime-free $q$ matroid $Z_{a-c,b-d}$, that is, $\beta(a-c,b-d;0,0)$. $Q$ is a product:  the power $q^{\binom{|(a-c)-(b-d)|}{2}}$  times a linear combination $Q_0$ of exactly {\it three} powers of $q$ (two of which may cancel or merge), with coefficients $\pm 1$,
$$
Q_0(a,b;c,d) = 1 + q^{|(a-c)-(b-d)|} - q^{max((a-c),(b-d))},
$$
so
$$
Q = q^{\binom{|(a-c)-(b-d)|}{2}}\;\;(1 + q^{|(a-c)-(b-d)|} - q^{max((a-c),(b-d))})
$$
The combined expression  $\beta(a,b;c,d) $ is equal to
$$
 (-1)^{(a-c)+(b-d)}\;\;\bin{a}{c}{q}\;\;\bin{b}{d}{q}\;\;q^{\binom{|(a-c)-(b-d)|}{2}}(1 + q^{|(a-c)-(b-d)|} - q^{max((a-c),(b-d))})
$$

The strange power of $q$, $q^{\binom{|(a-c)-(b-d)|}{2}}$ has a simple explanation. We know that the alternating sum, for any positive integer $n$, of classical binomial coefficients is zero:
$$\sum_{i=0}^n (-1)^i \binom{n}{i}\;=\;0$$
The corresponding formula for Gaussian binomial coefficients requires an extra power of $q$ in each summand.
The recursion 
$$ \bin{i-1}{j}{q} = \bin{i}{j}{q} - q^{i-j} \bin{i-1}{j-1}{q}$$
permits us to write the Gaussian binomial coefficients $\bin{i-1}{j}{q}$ as alternating sums of
the coefficients $\bin{i}{k}{q}$ for $k>j$.

\begin{theorem}
$$\sum_{i=0}^{n}  (-1)^i  q^{   \binom{i}{2} } \bin{n}{i}{q}    =  \delta_0^n.$$
\end{theorem}
\begin{proof}
For $ i>0$,
$$\begin{array}{ccc}
0=\bin{n-1}{n}{q}&=& \bin{n}{n}{q} -q^0 \bin{n-1}{n-1}{q}\\
\bin{n-1}{n-1}{q} &=& \bin{n}{n-1}{q} -q^1 \bin{n-1}{n-2}{q}\\
\bin{n-1}{n-2}{q} &=& \bin{n}{n-2}{q} -q^2 \bin{n-1}{n-3}{q}\\
\bin{n-1}{n-3}{q} &=& \bin{n}{n-3}{q} -q^3 \bin{n-1}{n-4}{q}\\
\dots\\
\bin{n-1}{3}{q} &=& \bin{n}{3}{q} -q^{(n-3)} \bin{n-1}{2}{q}\\
\bin{n-1}{2}{q} &=& \bin{n}{2}{q} -q^{(n-2)} \bin{n-1}{1}{q}\\
\bin{n-1}{1}{q} &=& \bin{n}{1}{q} -q^{(n-1)} \bin{n-1}{0}{q}\\
\bin{n-1}{0}{q} &=& \bin{n}{n}{q} -q^0 \bin{n-1}{-1}{q} = \bin{n}{n}{q}\\
\end{array}$$

So, making this string of substitutions, 
$$\begin{array}{ccl}
0 &=& \bin{n-1}{n}{q}\\
   &=& \bin{n}{n}{q} - q^{0} \bin{n-1}{n-1}{q}\\ 
   &=& \bin{n}{n}{q} - q^{0}\bin{n}{n-1}{q} + q^{0+1}\bin{n-1}{n-2}{q}\\ 
   &=& \bin{n}{n}{q} - q^{0}\bin{n}{n-1}{q} + q^{0+1}\bin{n}{n-2}{q}- q^{0+1+2}\bin{n-1}{n-3}{q}\\
   &=& ...  \\
   &=& \bin{n}{n}{q} - q^{\binom{1}{2}}\bin{n}{n-1}{q} + q^{\binom{2}{2}}\bin{n}{n-2}{q} - q^{\binom{3}{2}}\bin{n-1}{n-3}{q} \\
&&+ \dots + (-1)^{n-2} q^{\binom{n-2}{2}}\bin{n}{2}{q} + (-1)^{n-1} q^{\binom{n-1}{2}}\bin{n}{1}{q}\\
&& + (-1)^{n} q^{\binom{n}{2}}\bin{n}{0}{q} + (-1)^{n-1} q^{\binom{n+1}{2}}\bin{n}{-1}{q}\\
\end{array}$$
where $\bin{n}{-1}{q} = 0$, so the last term in this sum is zero.
\end{proof}

We will need an extension of Theorem 4 to include more general alternating sums of $q$-binomial coefficients that are also equal to zero.
The first step is to extend the  Pascal triangle of classical binomial coefficients $\binom{n}{k}$ to a domain in which $n$ may be negative. See Donald Knuth's ``Concrete Mathematics'' [K], page 164, equation (5.14).
\begin{lemma} For all integers $n$ and all non-negative integers $k\leq n$,
$$
\binom{-n}{k}\;=\;(-1)^k\;\binom{n+k-1}{k}
$$
\end{lemma}
\begin{proof} The formula is valid for $n=0$
Assume that the stated formula is valid for values of $n$ smaller than some non-negative integer value $m$.
Using the above formula, we find that
$$\begin{matrix}
\binom{-(m+1)}{k}\;+\;\binom{-(m+1)}{k-1}&\;=\;&(-1)^k\;(\binom{(m+k}{k}\;-\;\;\binom{(m+k-1}{k-1})\\
&\;=\;&(-1)^k\;\binom{m+k+1}{k}\\&\;=\;&\binom{-m}{k},\\
\end{matrix}$$
so the proposed definition for $\binom{-n}{k}$ is indeed that which obeys the Pascal recursion relation.
\end{proof}

\begin{conjecture}
For any non-negative value of $n$, and for any integer value of $s$ (the shift) with $0\leq s<n$,
$$
\sum^n_{i=0} \; (-1)^i \; q^{\binom{i-s}{2}} \; \bin{n}{i}{q} = \delta^n_0
$$
\end{conjecture}

The next step is to prove that the operators $\alpha$ and $\beta$ are inverse to one another in the incidence algebra of the ``rectangular'' poset of pairs of non-negative integers, with $(c,d)\leq (a,b)$ iff $c\leq a$ and $d\leq b$, so that the action of $\alpha$, then $\beta$ on a Tutte polynomial will convert it to the corresponding rank generating function, then back again to the original Tutte polynomial.

In a sense, we already "know" that the conversion processes are  inverse to one another, simply because
the Python program we have written to carry out these operations works on small examples. The processes are linear over the ring of $q$-polynomials. If anything were incorrect, these calculations would {\it never} work, not even for small examples. But that's practice, not theory.

It would suffice to show that the formula $\beta$ is correctly derived from the algorithm (recursive subtraction)
used to discover it. For the moment it seems preferable simply to prove that the operators are inverse to one another, i.e.: that the formula $\beta$ (found by experimentation) is inverse to the proven formula $\alpha$ for passage from the Tutte polynomial to the rank generating function. Here is a sketch of the beginning of a proof.

\begin{proof}
We define an $(a-e)\times(b-f)$ matrix $M$ of $q$-polyomials $$M_{i,j}(a,b;e,f)\;=\;\alpha(a,b;c,d)\;\beta(c,d;e,f),$$ where $i=c-e$ and $j=d-f$, and $M_{i,j}$ is the contribution to the concatenation product of $\alpha$ with $\beta$ due to the intermediate pair 
$(c,d) = (e-i,f+j)$, with $(e,f)\leq (e+i,f+j)\leq (a,b)$.
Thus
$$
M_{i,j}(a,b,e,f) = \bin{a}{c}{q} \bin{b}{d}{q} \bin{c}{e}{q} \bin{d}{f}{q}
q^{(a-c)(b-d)} (-1)^{i+j}\;q^{(a-e-i)(b-f-j)}\;q^{\binom{|i-j|}{2}}\;(1+q^{|i-j|}-q^{max(i,j)})
$$
We can simplify this expression by observing that
$$\begin{matrix}
\bin{a}{c}{q} \bin{b}{d}{q} \bin{c}{e}{q} \bin{d}{f}{q}
 = \frac{\faq{a}}{\faq{c}\faq{a-c}} \frac{\faq{c}}{\faq{e}\faq{c-e}} \frac{\faq{a-e}}{\faq{a-e}}\\
\phantom{}\\
 = \frac{\faq{a}}{\faq{e}\faq{a-e}} \frac{\faq{a-e}}{\faq{c-e}\faq{a-c}} 
= \bin{a}{e}{q} \bin{b}{f}{q} \bin{a-e}{i}{q} \bin{b-f}{j}{q},\\
\end{matrix}$$
thus revealing a sizeable common factor $F = \bin{a}{e}{q} \bin{b}{f}{q}$, independent of $i$ and $j$. Let $N_{i,j}(a,b;e,f)$ be the quotient of $M_{i,j}(a,b;e,f)$ by this common factor $F$: 
$$\begin{matrix}
N_{i,j}(a,b;e,f) &=& \bin{a-e}{i}{q} \bin{b-f}{j}{q} 
(-1)^{i+j}\;q^{(a-e-i)(b-f-j)}\;q^{\binom{|i-j|}{2}}\;(1+q^{|i-j|}-q^{max(i,j)})\\
&=&N_{i,j}(a-e,b-f,0,0)\hfill\\ 
\end{matrix}$$
In what follows, we write $N_{i,j}(a,b;0,0)$ more simply, as $N_{i,j}(a,b)$.
$$
N_{i,j}(a,b) \;=\; \bin{a}{i}{q} \bin{b}{j}{q} 
(-1)^{(i+j)}\;q^{(a-i)(b-j)}\;q^{\binom{|i-j|}{2}}\;(1+q^{|i-j|}-q^{max(i,j)})
$$
It thus suffices to prove that, for any non-negative integers $a$ and $b$, the sum of the $q$-polynomial entries in the matrix $N_{i,j}(a,b)$
is equal to the product $\delta_0^a\;\delta_0^b$ of Kronecker deltas.

\end{proof}

There is a path open toward a reasonable proof of this equality. 
\begin{enumerate}
\item[Stage 1.] The reduced matrix of entries $N_{i,j}$ for $a,b,e,f$ is equal to that for $a-e,b-f,0,0$, 
so we need consider only the instances $a,b,0,0$. {\it(Proof:)} The second formulation of $N_{i,j}$, above, is unchanged when $a$ and $e$ are replaced by $a-e$, and $0$, respectively.\\
\item[Stage 2.] In the instance $a,b,0,0$, all row sums are equal to zero whenever $a>b$, and all column sums are equal to zero whenever $a<b$. {\it(proof required)}\\
\item[Stage 3.] Say $a=b$. All column sums $C_i$ are divisible by the final column sum $C_a$. When this common factor is removed, the reduced column sums are {\it(proof required)}
$$
q^0\;\bin{a}{0}{q} \;-\;q^0\;\bin{a}{1}{q} \;+\;q^{\binom{2}{2}}\bin{a}{2}{q}\;-\;\dots\;
-\;q^{\binom{a-1}{2}}\bin{a}{a-1}{q}\;+\;q^{\binom{a}{2}}\bin{a}{a}{q},
$$
the sum being equal to zero, by Conjecture 2.
\item[Stage 4.] Say $a>b$. Let $p_i$, for $i$ non-negative, be the sequence $\binom{i}{2}$. We must extend this sequence into the negative index domain, using the recursion relation for the Pascal triangle, $\binom{n}{i}\;=\;\binom{n-1}{i-1}+\binom{n-1}{i}$
and the initial condition $\binom{0}{i}=\delta^i_0$ (the Kronecker delta). This yields the general formula valid for all integer values of $n$:
$$
\binom{n}{k}\;=\;(-1)^k\;\binom{(k-n)-1}{k}.
$$
 in a rather strange way, setting $p_i = \binom{i+1}{2}$ for $i<0$.
Then, for all $i$ in the interval $0\leq i\leq a$, the $i^{th}$ column sum $C_i$ is divisible by the {\it truncated} final column sum $C_a$. 
By "truncated" we mean "with all factors of $q$ removed". And when this
common factor is removed, the remaining $q$-polynomial $C_i/C_a$ is the sum of two polynomials:
$$\begin{matrix}
C_0/C_a&\;=\;&q^{p_a}\bin{a}{0}{q}&&\\
C_1/C_a&\;=\; & -   q^{p_{b-i}}\bin{a-1}{1}{q} &-&   q^{p_{b-i+1}}\bin{a-1}{0}{q}\\
C_2/C_a&\;=\; & +   q^{p_{b-i}}\bin{a-1}{2}{q} &+&   q^{p_{b-i+1}}\bin{a-1}{1}{q}\\
\dots\\
C_{a-2}/C_a&\;=\; & +   q^{p_{b-i}}\bin{a-1}{a-2}{q} &+&   q^{p_{b-i+1}}\bin{a-1}{a-3}{q}\\
C_{a-1}/C_a&\;=\; & - q^{p_{b-i}}\bin{a-1}{a-1}{q} &-&   q^{p_{b-i+1}}\bin{a-1}{a-2}{q} \\
C_a/C_a&\;=\;&&&\bin{a-1}{a-1}{q}\\
\end{matrix}$$
both columns in this display are alternating sums of $q$-binomial coefficients, with  multipliers ($q^{p_k}$). {\it(A suitable extension of Conjecture 2 is required, using this sequence $p_k$ with negative indices. The proof of such a theorem may resemble that for Conjecture 2 if we simply extend the definition of the $q$-Pascal triangle into the negative domain, using the same recurrence relation.)}. 
\item[Stage 5.] For square matrices, with $a=b$, the matrix is symmetric. All column sums are divisible by the final column
sum, for $i=a$, and when this common factor is removed, the reduced sums are exactly an alternating sum of $q$-binomial coefficients $\bin{a}{i}{q}$, with the required multipliers $q^{\binom{i}{2}}$, so the sum is zero by Theorem 4. {\it(Proof required for the reduction to this sum of the formula for $N_{i,j}$.)}
\end{enumerate}

\section{Partitions of a $q$-matroid into prime-free minors}

The correct numbers of prime-free minors of each length and activities can be obtained by the conversion process we have been using to pass from the rank generating function to the Tutte polynomial. It seemed for a while that any maximal such partition (maximal in the lattice of partitions into minors) in which all minors are prime-free, would be a partition leading to the Tutte polynomial. This is false, as the example below will indicate. It may still be true that any partition into a {\it minimal number of} prime-free minors will lead to the Tutte polynomial. (That remains to be proven.) There may, however, be maximal partitions (that is, in which no two partition parts can be merged) that do not have a minimal number of parts, as the following example, on the uniform $q$-matroid $U_{2,1}$, shows.

The bottom flat $\bo$ will be clopen in a minor of length two, with internal activity 2. The top flat $\tp$ will be clopen in a minor of length 1, with external activity 1. Say these two prime-free minors are $[\bo,L001]$ and $[L011,\tp]$. Then the remaining prime-free minors will form a maximal covering of the graph on the left of Figure~10 by disjoint edges. There are two such maximal coverings, shown center and right. But the choice on the right has a smaller number of partition parts, so it is maximal in the strong sense, even though it cannot be formed by merging two vertices of the center subgraph into an edge.

In Figure~10, we show the proposed clopen flats as boxed, the proposed added minors,  in addition to $[\bo,L001]$ and $[L011,\tp]$ as red edges joining a projective point to a projective line. In each such case of minors of height 1, it will be the point which is clopen.

\begin{figure}[h] 
\centering\includegraphics[scale=.64]{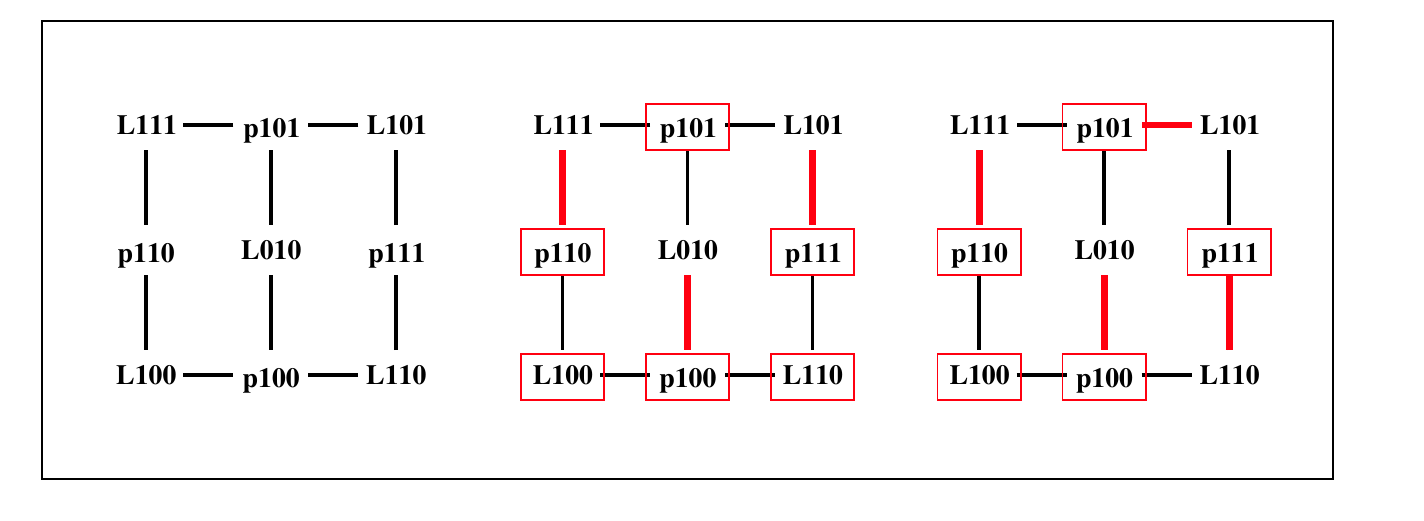} 
\caption{Maximal partitions, the Tutte partition on the right.}
\end{figure}

\section{Matroids on more general support lattices}

Significant research has been devoted to finding a way to consider Boolean algebras as lattices of subspaces of a vector space over the {\it  one-element field}. There is no such field, so this is a ``way of speaking''. But Boolean algebras $B(S) = 2^S$, as lattices, are complemented modular, but are irreducible only when $S$ has no more than one element. So if we were to define matroids on support lattices that are arbitrary complemented modular lattices, the study of matroids and of $q$ matroids become aspects of the same subject. In fact, this is completely straight forward. In Figure 11, we show the twelve matroids constructable on a support lattice $L$ that is the direct product of $L_{2,2}$ with the Boolean algebra of a 1-point set. Each row of the diagram is a pair of dual matroids on support lattice $L$. In the geometric figures, the support flats that are independent are marked in blue, dependent in yellow, and the outer circle stands for the entire plane.

\begin{figure}[h] 
\centering\includegraphics[scale=.5]{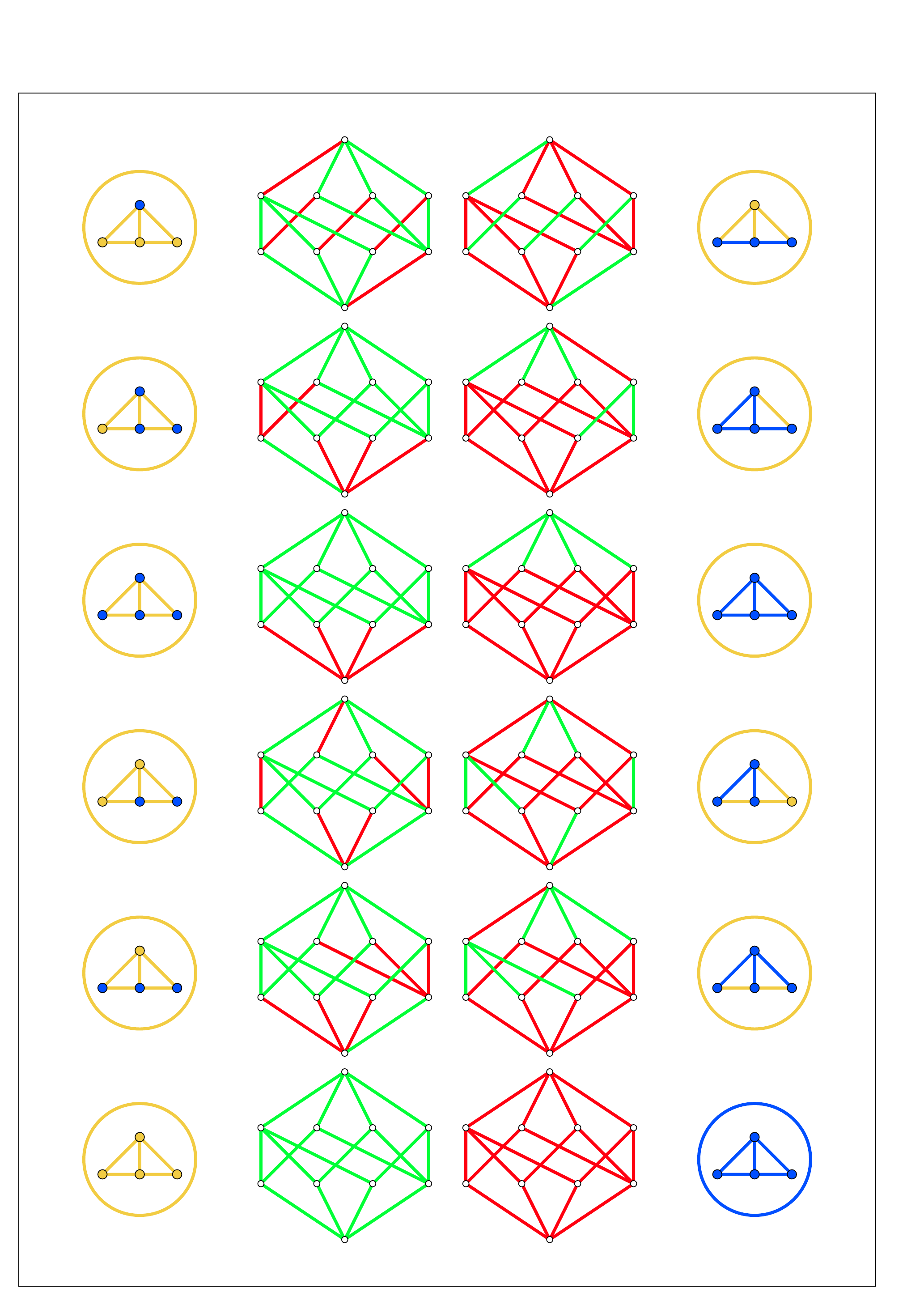} 
\caption{The twelve isomorphism classes of matroids over this support lattice.}
\end{figure}

An axiom system for such matroids on a {\it modular} lattice $L$  is exactly that given above for $q$-matroids: every diamond must be either zero, one, mixed, or prime. These structures do not have
the usual matroid properties, with their several cryptomorphic axiom systems, until we insist that the support lattice $L$ be {\it relatively complemented}, that is, that every minor is complemented. To see how bad things can get for non-complemented support lattices, consider the following example. If $L$ is the chain $0 < 1 < 2 < 3 < 4$, and if covers $[0,1]$ and $[2,3]$ are red, $[1,2]$ and $[3,4]$ are green, then $1$ is maximal independent, while $3$ is minimal spanning, so there are no bases.

Non-complemented support lattices do, however, give rise to weaker (non-matroidal) quotient structures such as quasi-matroids [F], selectors [C4], anti-matroids [JW], [E], greedoids [K-L].

\clearpage

\section{Bibliography}


[JP] R.~Jurrius and R.~Pellikaan, Defining the $q$-Analogue of a Matroid, \hfill\\ arXiv:1610.09250, October 2016.

[Co] Henry Cohn, Projective Geometry over $F_1$, and the
Gaussian Binomial Coefficients, (add reference)

[KC] Victor Kac, Pokman Cheung, Quantum Calculus, Springer Universitext, 2001.

[K] Donald Knuth, Concrete Mathematics\\
https://www.csie.ntu.edu.tw/\~\;r97002/temp/Concrete\%20Mathematics\%202e.pdf 

[B] John Baez, This week's finds in mathematical physics, "Weeks" 183-188, http://math.ucr.edu/home/baez/week183.html  

[LW] J.H. van Lint, R.M. Wilson: A Course in Combinatorics. Cambridge University Press, second edition, 2001, Chapter 24.

[C1] H.~Crapo, On the Theory of Combinatorial Independence, Doctoral Thesis, M.I.T., 1964, 228 pp.

[C2] H.~Crapo, Lattice Differentials, and the Theory of Combinatorial Independence, Research Reports, Northeastern University, 1964.

[C3] H.~Crapo, The Tutte Polynomial, Aequationes Mathematicae [{\bf 3}] (1969), 211-229.

[T1] W.~T.~Tutte, On Dichromatic Polynomials, J. Combin. Theory [{\bf 2}], 1967, 301-320. 

[F]  U.~Faigle, Geometries on partially-ordered sets, J. Combin. Theory Ser.B [{\bf 28}] (1980), 26-51.

[JW] R.~Jamison-Waldner, Copoints in antimatroids, in ``Proceedings, 11th Southeast Conference on Combinatorics, Graph Theory, and Computing'', Congr. Numer. [{\bf 29}] (1980), 535-544.

[E] P.~H.~Edelman, Meet-distributive lattices and the anti-exchange closure, Algebra Universalis [{\bf 10}] (1980), 290-299.

[K-L] B.~Korte and L.~Lovasz, ``Greedoids, a Structural Framework for the Greedy Algorithm'', Bonn University report 82230-OR; in ``Proceedings,Silver Jubilee Conference on Combinatorics'',1982.

[C4] H.~Crapo, Selectors: A Theory of Formal Languages, Semimodular Lattices, and Branching and Shelling Processes, Advances in Math. [{\bf 54}] (1984), 233-277.

\end{document}